\documentclass[12pt]{amsart}
\usepackage{graphicx}
\usepackage{amsmath, enumerate}
\usepackage{amsfonts, amssymb, amsthm, xcolor, stmaryrd}
\usepackage{tikz, tikz-cd, caption, tabu, longtable, mathrsfs, xtab, centernot, mathtools}
\usepackage{leftidx}
\usepackage{dsfont, cite}
\usepackage[textwidth=35mm,colorinlistoftodos]{todonotes}
\usepackage[toc,page]{appendix}
\usepackage{hyperref}

\numberwithin{equation}{section}
\theoremstyle{plain}

\newtheorem{thm}{Theorem}[section]
\newtheorem{lemma}[thm]{Lemma}
\newtheorem{prop}[thm]{Proposition}
\newtheorem{cor}[thm]{Corollary}

\theoremstyle{definition}

\newtheorem{exmp}[thm]{Example}

\theoremstyle{remark}
\newtheorem{rmk}[thm]{Remark}

\newcommand{\Mp}{{\mathrm{Mp}}}
\newcommand{\Nm}{{\mathrm{Nm}}}

\newcommand{\SO}{{\mathrm{SO}}}

\newcommand{\sgn}{\mathrm{sgn}}

\newcommand{\Bb}{\mathbb{B}}
\newcommand{\Cb}{\mathbb{C}}

\newcommand{\Hb}{\mathbb{H}}

\newcommand{\Nb}{\mathbb{N}}
\newcommand{\Qb}{\mathbb{Q}}

\newcommand{\Rb}{\mathbb{R}}
\newcommand{\Zb}{\mathbb{Z}}

\newcommand{\ebf}{\mathbf{e}}

\newcommand{\af}{\mathfrak{a}}

\newcommand{\bfrak}{\mathfrak{b}}
\newcommand{\df}{\mathfrak{d}}

\newcommand{\Dc}{{\mathcal{D}}}

\newcommand{\Oc}{\mathcal{O}}

\newcommand{\varep}{\varepsilon}
\newcommand\smod[1]{\ (\operatorname{mod} #1)}

\newcommand{\lp}{\left (}
\newcommand{\rp}{\right )}

\newcommand{\tc}{\tilde{c}}

\newcommand{\tD}{\tilde \Delta}
\newcommand{\e}{\epsilon}
\newcommand{\hv}{\hat{\vartheta}}
\newcommand{\tv}{\tilde{\vartheta}}
\newcommand{\tbe}{\tilde{\beta}}
 \usepackage{multirow}
\begin{document}
\title[]{Mock Maass Forms Revisited}
\author{Yingkun Li}
\author{Christina Roehrig}
\address{Fachbereich Mathematik,
Technische Universit\"at Darmstadt, Schlossgartenstrasse 7, D--64289
Darmstadt, Germany}
\email{li@mathematik.tu-darmstadt.de}
\email{roehrig@mathematik.tu-darmstadt.de}
\date{\today}
\begin{abstract}
  In this paper, we use theta integrals to give a different construction of mock Maass forms studied by Sander Zwegers.
  With this method, we construct new real-analytic modular forms, whose Fourier coefficients are logarithms of algebraic numbers in a real quadratic field.
\end{abstract}
\maketitle
 \makeatletter
 \providecommand\@dotsep{5}
 \def\listtodoname{List of Todos}
 \def\listoftodos{\@starttoc{tdo}\listtodoname}
 \makeatother


\section{Introduction}
In \cite{Co88}, Cohen shows that the function
\begin{equation*}
	\varphi_0(\tau):=\sqrt{v}\sum_{\substack{n\in\Zb\\n\neq 0}} T(n)\,e^{2\pi inu/24}\,K_0(2\pi|n|v/24)\quad(\tau=u+iv\in\Hb)
\end{equation*}
is a weight 0 Maass form, i.e.\ it transforms as a modular form of weight 0 on $\Gamma_0(2)$ up to some explicitly given multiplier and it is an eigenfunction of the weight 0 Laplace operator $\Delta_\tau$ defined in \eqref{eq:tD}. Here, $K_0$ denotes a modified Bessel function of the second kind, which makes immediately clear that the second property holds (see \eqref{eq:K0} for the exact definition). 
As an arithemtic function, $T(n)$ is defined for any $n\equiv 1\smod{24}$ and describes the excess of the number of inequivalent solutions $(x,y)$ of Pell's equation $x^2-6y^2=n$ with $x+3y\equiv\pm 1\smod{12}$ over the number of inequivalent solutions with $x+3y\equiv\pm 5\smod{12}$.

One can also establish the relation
\begin{equation*}
	\sum_{\substack{n\in\Zb\\n\equiv 1\smod{24}}} T(n)\,q^{|n|/24}=q^{1/24}\sigma(q)+q^{-1/24}\sigma^\ast(q),
\end{equation*}
where $\sigma$ and $\sigma^\ast$ are the $q$-hypergeometric functions defined by
\begin{align*}
	\sigma(q)&:=\sum_{n\geq 0}\frac{q^{n(n+1)/2}}{(1+q)\dots (1+q^n)}=1+\sum_{n\geq 0}(-1)^nq^{n+1}(1-q)\ldots(1-q^n)\\
	&=1+q-q^2+2q^3-2q^4+q^5+\ldots,
\end{align*}
respectively
\begin{align*}
	\sigma^\ast(q)&:=2\sum_{n\geq 1}\frac{(-1)^nq^{n^2}}{(1-q)(1-q^3)\dots (1-q^{2n-1})}=-2\sum_{n\geq 0}q^{n+1}(1-q^2)(1-q^4)\ldots(1-q^{2n})\\
	&=-2q-2q^2-2q^3+2q^7+2q^8+2q^{10}+\ldots .
\end{align*}
These functions can be rewritten as indefinite theta functions.
Zwegers \cite{Zw12} used this to put Cohen's example in a more general framework and interpreted $\varphi_0$ as the ``harmonic'' part $\Phi^{c_1, c_2}_{a,b}$ of a real-analytic modular form $\hat\Phi^{c_1, c_2}_{a,b}$, where harmonic means that the function vanishes under the modified Laplace operator
\begin{equation}
	\label{eq:tD}
	\tD := \Delta_\tau + \frac{1}{4},~
	\Delta_\tau := v^2 (\partial_u^2 + \partial_v^2).
\end{equation}
In general, Zwegers showed that $\tD \hat\Phi^{c_1, c_2}_{a,b}$ is the difference of two theta series $\theta^{(1, 1)}$ depending on negative lines $c_1$ and $c_2$ respectively (see Theorem \ref{thm:Zw-diff} for details).

In this paper, we will use the theta integral approach as in \cite{FK17, CL20} to extend these results and construct a preimage $\tv$ of \textit{one} such theta function $\theta^{(1, 1)}$ 
under the operator $\tD$.
Furthermore, the harmonic part Fourier coefficients of $\tv$ are logarithms of algebraic numbers in real quadratic fields, similar to the ones studied in \cite{CL20}.

\begin{thm}
  \label{thm:main-intro}
  Let $L$ be an anisotropic lattice of signature $(1, 1)$ and $\theta^{(1, 1)}_{L+h}(\tau, t)$ the theta function defined in \eqref{eq:theta11} for $h \in L^*/L$.
  Then there exists $\tv_{L+h}(\tau, t)$ such that
  \begin{equation}
    \label{eq:diff-eq-intro}
  \tD \tv_{L+h}(\tau, t) = \theta^{(1, 1)}_{L+h}(\tau, t)    
  \end{equation}
 for all $t > 0$, and whose harmonic part coefficients $\tc_t$ are given in \eqref{eq:cLam}.
  Furthermore, we have $\Phi^{c_1, c_2}_{h, 0} = \hv_{L+h}^{c_1, c_2}(\tau) = \tv_{L+h}(\tau, t_2) - \tv_{L+h}(\tau, t_1)$ when $c_i = c(t_i)$ via \eqref{eq:c} with $\hv$ defined in \eqref{eq:int-c1-c2}. 
  \end{thm}
\begin{rmk}Similar results should also hold in  the isotropic case, as in \cite{Li17}. 
\end{rmk}

\begin{rmk}
  In Examples  \ref{ex:cohen2} and  \ref{ex:non-trivial}, we show that our result goes beyond the work of Zwegers and produces a new type of modular form. 
\end{rmk}
The paper is organized in the following way. We recall preliminary notions of theta functions in Section 2.
Section 3 contains our main construction. The first half rephrases the construction of Zwegers. The second half extends it and proves Theorem \ref{thm:main}, which implies Theorem \ref{thm:main-intro}.
The last section proposes a generalization of the construction in Section 3, and some future directions.

{\bf Acknowledgement}: 
The authors are supported by the Deutsche Forschungsgemeinschaft (DFG) through the Collaborative Research Centre TRR 326 \emph{Geometry and Arithmetic of Uniformized Structures}, project number 444845124.
\section{Preliminaries}
\label{sec:2}
In this section, we recall preliminary notions of quadratic space, unary theta functions and theta kernels that will be used in constructing mock Maass forms later. As usual, we denote $\ebf(z):=e^{2\pi iz}$.

Let $L$ be an even, integral lattice of rank $n$ with a quadratic form $Q$ and associated bilinear form $B(\cdot, \cdot)$, and denote $L^* \subset V := L \otimes \Qb$ its dual lattice.
When $Q$ is positive definite, we have the theta function
\begin{equation}
  \label{eq:thetaLh}
  \theta_{L+h, p}(\tau) :=   \sum_{x \in L+h} p(x) \ebf(Q(x)\tau)\quad (h \in V),
\end{equation}
where $p: V (\Rb) \to \Cb$ is a harmonic polynomial of degree $d$.
We omit the subscript $p$ when it is the constant 1.
The theta function is a modular form of weight $d + n/2$ and a congruence subgroup of $\Mp_2(\Rb)$.
The vector-valued modular form
\begin{equation}
  \label{eq:ThetaL}
\Theta_{L, p}(\tau) = (\theta_{L+h, p}(\tau))_{h \in L^*/L}  
\end{equation}
transforms with respect to the Weil representation $\rho_L$ on $\Mp_2(\Zb)$. 
See \cite{Bo98, Sc09} for more information about $\rho_L$.
When $Q$ is negative definite, we set
\begin{equation}
  \label{eq:theta-neg}
  \theta_{L+h, p}(\tau) := v^{d + n/2} \theta_{L^-+h, p}(-\overline\tau),
\end{equation}
where $L^-$ denotes the lattice $L$ with the quadratic form $-Q$. 
When $L$ has rank 1, the polynomial $p$ is necessarily linear and we denote
\begin{equation}
  \label{eq:unary}
  \theta_{L+h}^{(\e)} := \theta_{L+h, x^\e},~
  \Theta_{L}^{(\e)} := \Theta_{L, x^\e}
\end{equation}
for $\e \in \{0, 1\}$.

When $Q$ is not definite, i.e.\ of signature $(n^+, n^-)$ with $n^\pm \ge 1$, one can attach to $L$ a two variable theta function in the following way.
Let $\Dc$ be the Grassmannian of negative $n^-$-planes in $V(\Rb) \cong  \Rb^{n^+} \oplus (\Rb^{n^-})^-$ with $\Rb^n$ the definite $\Rb$-quadratic space with quadratic form
$Q((x_1, \dots, x_n)) := \sum_{1 \le i \le n} x_i^2/2$.
For each $Z \in \Dc$ and $\lambda \in V(\Rb)$, let $Z^\perp \subset V(\Rb)$ be its orthogonal complement and $\lambda_Z, \lambda_{Z^\perp}$ the projections of $\lambda$ such that $Q(\lambda)=Q(\lambda_{Z^\perp})+Q(\lambda_{Z})$. Then the Siegel theta function given by
\begin{equation}
  \label{eq:Siegel}
  \theta_{L+h}(\tau, Z) :=
v^{n^-/2}
  \sum_{\lambda \in L + h}
  \ebf\lp Q(\lambda_{Z^\perp}) \tau +  Q(\lambda_{Z}) \overline\tau\rp
\end{equation}
is a real-analytic modular form in $\tau$ of weight $\frac{n^+-n^-}{2}$.
More generally, one can add a harmonic polynomial $p$ of degree $d^+$, resp.\ $d^-$, in the coordinates of ${Z^\perp} \cong \Rb^{n^+}$, resp.\ $Z \cong(\Rb^{n^-})^-$, to define $\theta_{L+h, p}$, which will be a real-analytic modular form of weight $\frac{n^+-n^-}{2} + d^+ - d^-$.

This is a special case of the theta series considered by Borcherds. In Theorem 4.1 in \cite{Bo98}, the explicit transformation behavior is given.
For example, when $V$ is indefinite of dimension 2, it will have signature $(1,1)$ and $V(\Rb)$ is isometric to the real hyperbolic plane $V_0(\Rb) = \Rb^2$ with quadratic form $Q_0$ given by
\begin{equation}
  \label{eq:Q0}
Q_0((x_1, x_2)) := x_1 x_2.
\end{equation}
Let $B_0$ be the associated bilinear form.
The symmetric space $\Dc$ of $V_0$ is parametrized by $\Rb^\times$ via
\begin{equation}
  \label{eq:c}
  \begin{split}
  c: \Rb^\times &\to \Dc \\
t &\mapsto \Rb \cdot Z^-_t,~
  \end{split}
\end{equation}
where we denote 
\begin{equation}
  \label{eq:Zt}
   Z^\pm_t :=    \lp {t}{}, \pm {t^{-1}}{} \rp \in V_0(\Rb).
\end{equation}
The space $\Dc$ has two  connected components, and we denote by $\Dc^+$ the component parametrized by $\Rb^\times_+$ under $c$.
For any $\lambda = (\lambda_1, \lambda_2) \in V_0(\Rb)$, we can write it as
\begin{equation}
  \label{eq:lambda-decomp}
  \lambda =   \lambda_t^+ Z_t^+ - \lambda_t^- Z_t^-,~
  \lambda_t^\pm
  :=  \frac{  B_0(\lambda, Z_t^\pm)}{{2}} 
=  \frac{ \pm  \lambda_1 t^{-1} + \lambda_2 t}{{2}} \in \Rb,
\end{equation}
which satisfy the equations
\begin{equation}
  \label{eq:lam-id}
  (  \lambda^+_t)^2 + (  \lambda^-_t)^2 = \frac{  \lambda^2_1 t^{-2} + \lambda^2_2 t^2}{{2}},~
  (  \lambda^+_t)^2 - (  \lambda^-_t)^2 =  \lambda_1\lambda_2 = Q_0(\lambda).
\end{equation}


For a lattice $L \subset V = V_0$, the theta function $\theta_{L+h, p}$ can be explicitly written down as
\begin{equation}
\label{eq:theta_explicit}
\theta_{L+h, p} (\tau, t) =
v^{d^-+1/2}
\sum_{
  \begin{subarray}{c}
\lambda \in L +h
  \end{subarray}}
p(\lambda_t^+, \lambda_t^-)
\ebf\lp
  (\lambda_t^+)^2 \tau -   (\lambda_t^-)^2 \overline{\tau} 
\rp.
\end{equation}
Note that the only harmonic polynomials on $V(\Rb) \cong \Rb^{1, 1}$ have degree $(d^+, d^-)$ with $d^\pm  \le 1$.
We denote
\begin{equation}
  \label{eq:theta11}
  \theta_{L+h}^{(d^+, d^-)}(\tau, t) :=   \theta_{L+h, x^{d^+} y^{d^-}}(\tau, t)
\end{equation}
where $(x, y) \in \Rb^{1, 1}$. 
When $t = t_0\in \Rb^\times$ satisfies $t_0^2 \in \Qb^\times$, we have a rational splitting $V = U \oplus U^\perp$ with $U$ positive definite such that $U^\perp(\Rb) = c(t_0)$, and $P \oplus N \subset L$ is a finite index sublattice with $P := L \cap U$ and $N := L \cap U^\perp$ definite lattices.
The theta function $\theta_{L+h}^{(d^+, d^-)}$ has the following corresponding splitting
\begin{equation}
  \label{eq:theta-split}
  \theta^{(d^+, d^-)}_{L+h}(\tau, t_0)=
  \sum_{\substack{\mu \in P^*/P,~ \nu \in N^*/N\\P\oplus N + (\mu, \nu) \subset L + h}}
  \theta^{(d^+)}_{P+ \mu}(\tau)   \theta^{(d^-)}_{N+ \nu}(\tau)
\end{equation}
for $d^\pm \in \{0, 1\}$.

Suppose $L$ is anisotropic. Then
\begin{equation}
  \label{eq:Laf}
 L = L_{\af, M} = M\af 
\end{equation}
for some fractional ideal $\af$ of a real quadratic field $F$ with quadratic form $Q_{\af, M} = \Nm/(AM)$,  $A = \Nm(\af)$ and $M \in \Nb$.
Its dual lattice is $\af\df^{-1}$ with $\df$ the different of $F$.
  The isometry between $V(\Rb)$ and $V_0(\Rb)$ is given by
\begin{equation}
  \label{eq:embedding}
  \begin{split}
\iota: V(\Rb) &\to V_0(\Rb)\\
\lambda &\mapsto \frac{1}{\sqrt{AM}} (\lambda,  \lambda').
  \end{split}
\end{equation}
Using this isometry, we can give the following explicit expression of
$\theta_{L + h, p}(\tau, t)$ as in \eqref{eq:theta_explicit}
\begin{equation}
  \label{eq:thetaLhp}
  \theta_{L+h, p}(\tau, t) = v^{d^- + 1/2} \sum_{\lambda \in L_{\af, M} + h}
  p\lp
  \frac{ \lambda t^{-1} + \lambda' t}{2\sqrt{AM}}, \frac{ -\lambda t^{-1} + \lambda' t}{2\sqrt{AM}}  \rp
  \ebf(Q_{\af, M}(\lambda)u) \ebf \lp -\pi \frac{(\lambda t^{-1})^2 + (\lambda' t)^2}{AM} \rp.
\end{equation}
The splitting in \eqref{eq:theta-split} happens when $t = t_0 \in \Rb^\times$ satisfies
\begin{equation}
  \label{eq:trat}
  t_0^2 \in F^1:= \{\alpha \in F^\times: \Nm(\alpha) = 1\}.
\end{equation}


For later purpose, we will also require the $K$-Bessel function $K_0(x)$, which is defined by
\begin{equation}
  \label{eq:K0}
K_0(x) := \int^\infty_0 e^{-x \cosh (T)} dT  ,~ x > 0
\end{equation}
and satisfies the second order differential equation $$\lp x\frac{\partial^2}{\partial x^2}+\frac{\partial}{\partial x}-x\rp K_0(x)=0,$$
with the asymptotics
\begin{equation}
  \label{eq:K0inf}
  K_0(x) \sim \sqrt{\pi/ (2x)} e^{-x}
\end{equation}
as $x \to \infty$.
For convenience, we also set
\begin{equation}
  \label{eq:K00}
  K_0(0) := 0
\end{equation}
and state a simple lemma that will be useful later.
\begin{lemma}
  \label{lemma:asymp}
  For any $a \in \Rb$, we denote $K_0(x; a) := \sgn(a) \int_{|a|}^\infty e^{-x \cosh(T)}dT$.
  Then
  \begin{equation}
    \label{eq:asymp}
    \lim_{x \to \infty} \frac{K_0(x; a)}{K_0(x)} = 0.
  \end{equation}
\end{lemma}

\begin{proof}
  The claim is trivial for $a = 0$ and we only need to prove it for $a > 0$.
  In that case, $b:= \sinh(a/2) > 0$ and a change of variable $u = \sqrt{x} \sinh(T/2), du = \sqrt{x} \cosh(T/2)\,dT/2 =\sqrt{u^2 + x}\, dT/2$ gives us
  \begin{align*}
\sqrt{x} e^x    K_0(x; a)
    &
    = \sqrt{x} \int_a^\infty e^{-2x \sinh(T/2)^2}dT
      = 2      \int^\infty_{\sqrt{x} b} e^{-2u^2} \frac{du}{\sqrt{u^2/x + 1}}
      < 2      \int^\infty_{\sqrt{x} b} e^{-2u^2} du.
  \end{align*}
  In particular, we have the bound
  \begin{equation}
    \label{eq:K0bound}
    |K_0(x; a)| < \frac{e^{-x(1+2b^2)}}{x b}.
  \end{equation}
  Taking the limit $x \to \infty$ and applying \eqref{eq:K0inf} finishes the proof.
\end{proof}

\section{Mock Maass Forms}
In this section, we will use the suitable integrals of Siegel theta kernels to construct higher depth analogues of the mock Maass forms constructed by Zwegers.
In particular, we will express Maass' and Zwegers' constructions as special examples (see Section \ref{subsec:mock1}).

\subsection{Mock Maass Theta Functions I}
\label{subsec:mock1}
Let $L = \iota(L_{\af, M})\subset V_0(\Rb)$
be the  even, integral, anisotropic lattice as in \eqref{eq:Laf}, and  $\Gamma_L \subset \SO(L)$ the subgroup
of the discriminant kernel consisting of  totally positive units in $\Oc_F^\times$ congruent to 1 modulo $M\sqrt{D}$. 
Then $\Gamma_L$ is isomorphic to $\Zb$.
%
Let 
$\varep_L > 1$ be its generator.
Given $t_2 >  t_1 > 0$ and denote $c_i := c(t_i) \in \Dc$, we consider
\begin{equation}
  \label{eq:int-c1-c2}
  \hv^{c_1, c_2}_{L + h} (\tau)
  := \int^{t_2}_{t_1} \theta_{L+h}(\tau, t) \frac{dt}{t}.
\end{equation}
This is a real-analytic modular function in $\tau = u + iv \in \Hb$.
When $t_2/t_1 \in \Gamma_L$, we can evaluate \eqref{eq:int-c1-c2} by unfolding.
For $h \not\in L$, we can unfold the integral to obtain
\begin{align*}
  \hv^{c_1, c_2}_{L+h}(\tau)
  &=\frac{\log(t_2/t_1)}{\log \varep_L}
    \int^{\varep_L}_1 \theta_{L+h}(\tau, t) \frac{dt}{t}\\
&= \sqrt{v}
\frac{\log(t_2/t_1)}{\log \varep_L}
\sum_{\lambda \in \Gamma_L \backslash L + h}
\ebf\lp
Q_0(\lambda)u
\rp
\int^\infty_{-\infty}
\ebf\lp
\frac{
(\lambda e^{-\nu})^2 + (\lambda' e^\nu)^2 
}{2}iv
\rp
d\nu\\
&= 
  2\sqrt{v}
  \frac{\log(t_2/t_1)}{\log \varep_L}
\sum_{\lambda \in \Gamma_L \backslash L + h}
\ebf\lp
Q_0(\lambda)u
\rp
\int^\infty_{0}
e^{-2\pi v |Q_0(\lambda)|
\cosh(T)}
dT.
\end{align*}
Using \eqref{eq:K0}, we can evaluate
\begin{equation}
  \label{eq:hv}
  \hv^{c_1, c_2}_{L+h}(\tau) 
= 
{2\sqrt{v}}
\frac{\log(t_2/t_1)}{\log \varep_L}
\sum_{\lambda \in \Gamma_L \backslash L + h}
\ebf\lp Q_0(\lambda) u \rp
K_0(2\pi v |Q_0(\lambda)|).
\end{equation}
When $h \in L$, we have the additional constant term
$$
{\sqrt{v}}
\frac{\log(t_2/t_1)}{\log \varep_L}
\int^{\varep_L}_1 \frac{dt}{t}
= {\sqrt{v}}
\log(t_2/t_1).
$$
This recovers the construction of Maass cusp forms by Maass \cite{Ma49}.
For $c_1 = c(1)$ and $c_2 = c(\varep_L)$, we denote
\begin{equation}
  \label{eq:vartheta}
  \vartheta_{L+h}(\tau) :=   \hv^{c_1, c_2}_{L+h}(\tau).
\end{equation}

More generally, we can evaluate the Fourier expansion of $\hv^{c_1, c_2}_{L+h}$ without unfolding as
\begin{align*}
  \hv^{c_1, c_2}_{L + h} (\tau)
  &=
\sum_{\lambda \in  L + h}
\int^{t_2}_{t_1}
    \phi_\tau(\lambda^+_t, \lambda^-_t) \frac{dt}{t}
=  {\sqrt{v}}
\sum_{\lambda \in L + h}
\ebf\lp
Q(\lambda)u\rp
\int^{t_2}_{t_1}
\ebf\lp
\lp  (\lambda^+_t)^2
+ (\lambda^-_t)^2 \rp
  iv
\rp
    \frac{dt}{t}
  \\
&=  {\sqrt{v}}
\sum_{\lambda \in L + h}
\ebf\lp
Q_0(\lambda)\tau\rp
\int^{t_2}_{t_1}
e^{-4 \pi v (\lambda^-_t)^2 }
  \frac{dt}{t}
  =  {\sqrt{v}}
\sum_{\lambda \in L + h}
\ebf\lp
Q_0(\lambda)\tau\rp
\int^{t_2}_{t_1}
e^{-\pi v B_0(\lambda, Z^-_t)^2 }
\frac{dt}{t}.
\end{align*}
This is  Zwegers' function $\hat\Phi^{c_1, c_2}_{h, 0}$ in Definition 2.3 of \cite{Zw12}, and we have shown that it possesses an integral representation given in \eqref{eq:int-c1-c2}.
For general $a, b \in L^*$, we can express $\hat\Phi^{c_1, c_2}_{a, b}$ as a linear combination of $\hv_{\tilde L + \tilde h}$'s for a suitable  sublattice $\tilde L \subset L$. See  \ref{ex:cohen} for an example.

Now define  for $c_i = c(t_i)$
\begin{equation}
  \label{eq:vt}
  \begin{split}
    \hv^{c_1, c_2, +}_{L+h}(\tau)
    &:=
      \sqrt{v}
      \sum_{\lambda \in L+h}
      a_{\lambda}(t_1, t_2)
      \ebf(Q_0(\lambda) u) K_0(2\pi |Q_0(\lambda)| v),    \\
    \varphi^{c_0}_{L+h}(\tau)
    &:=  \sqrt{v} \sum_{\lambda \in L + h} \beta_{t_0}(\lambda \sqrt{v}) \ebf(Q_0(\lambda)u)
  \end{split}
\end{equation}
with the coefficients $a_\lambda$ defined by
\begin{equation}
	\label{eq:am}
	a_{\lambda}(t_1, t_2) :=
	\begin{cases}
		\frac{1}{2} \lp 1 - \sgn(\lambda_{t_1}^{+} \lambda_{t_2}^{+}) \rp & \text{ if } Q_0(\lambda)<0,\\ 
		\frac{1}{2} \lp 1 - \sgn(\lambda_{t_1}^{-} \lambda_{t_2}^{-}) \rp & \text{ if } Q_0(\lambda)>0,\\ 
		\log (t_2/t_1) & \text{ if } \lambda = 0,
	\end{cases}  
      \end{equation}  
and the function $\beta_t$ defined by
\begin{equation}
  \label{eq:beta}
  \beta_{t_0}(w) :=
  \begin{cases}
e^{-2\pi Q_0(w)}    \int^\infty_{t_0} e^{-4\pi (w_t^-)^2} \frac{dt}{t} &\text{ if } w_{t_0}^- w_{t_0}^+ > 0,\\
    - e^{-2\pi Q_0(w)}\int_{0}^{t_0} e^{-4\pi (w_t^-)^2} \frac{dt}{t}&\text{ if } w_{t_0}^- w_{t_0}^+ < 0,\\
0 &\text{ if } w_{t_0}^- w_{t_0}^+= 0,   
  \end{cases}
\end{equation}
a function on $V_0(\Rb) \backslash \{0\}$.

From the definition of $\beta_t$ in \eqref{eq:beta}, one can check that
\begin{equation}
  \label{eq:mock-maass-check}
  \begin{split}
e^{-2\pi Q_0(w) }\int^{t_2}_{t_1}
e^{-4 \pi (w^-_t)^2 }
  \frac{dt}{t}
    &  =
      a_\lambda(t_1, t_2) K_0(2\pi |Q_0(w)|)
 +      \beta_{t_1}(w)  - \beta_{t_2}(w)
  \end{split}
\end{equation}
for $w \in V_0(\Rb)$.
Using this relation,
Zwegers decomposed $\hat\Phi_{h, 0}^{c_1, c_2} = \hat\vartheta^{c_1, c_2}_{L+h}$ into a harmonic and a non-harmonic part in the following way.
\begin{thm}
  \label{thm:Zw12}[Theorem 2.6 in \cite{Zw12}]
  In the notations above, we have
  \begin{equation}
      \label{eq:har-decomp}
  \hv^{c_1, c_2}_{L + h}(\tau) = \hv^{c_1, c_2,+}_{L+h}(\tau) + \varphi^{c_1}_{L+h}(\tau) - \varphi^{c_2}_{L+h}(\tau) .
  \end{equation}
\end{thm}
\begin{rmk}
  \label{rmk:alam}
  It is easy to check that the coefficients $a_\lambda$ have the following simpler form
  \begin{equation}
    \label{eq:alam2}
    a_\lambda(t_1, t_2)=
    \begin{cases}
      1 & \text{if } \lambda \neq 0 \text{ and } t_1^2 < \left|\frac{\lambda_1}{\lambda_2}\right| < t_2^2,\\
      \frac{1}{2} & \text{if } \lambda \neq 0 \text{ and } t_i^2 = \left|\frac{\lambda_1}{\lambda_2}\right| \text{ for } i \in \{1, 2\},\\
      \log\frac{t_2}{t_1} & \text{if } \lambda = 0,\\
      0& \text{otherwise.}
    \end{cases}
  \end{equation}
\end{rmk}

The function $\hv_{L+h}^{c_1, c_2}$ also satisfies nice properties under differential operators.
Suppose $w = (w_1, w_2)$ with $Q_0(w) = w_1w_2 \neq 0$,
then
$$
\sgn(w_{t}^- w_{t}^+)
= \sgn((w_2t)^2 - (w_1t^{-1})^2)
= \sgn(\log(t^2 |w_2/w_1|))
$$
and we can rewrite
\begin{equation}
  \label{eq:beta2}
  \begin{split}
      \beta_{t_0}(w)
&  =
\sgn((w_2t_0)^2 - (w_1t_0^{-1})^2)
  \int^\infty_{t_0|w_2/w_1|^{1/2}} e^{-\pi |Q_0(w)| (t^2+t^{-2})} \frac{dt}{t}\\
&  = \frac12 K_0(2\pi |Q_0(w)|, \log (t_0^2 |w_2/w_1|)).
  \end{split}
\end{equation}
A straightforward calculation gives the identity
\begin{align*}
	(v\partial_v^2
	+ \partial_v
	- x^2 v   )\lp K_{0}(xv;a)\rp= x\sinh(a)\, e^{-xv\cosh(a)}.
\end{align*}
Using \eqref{eq:beta2} it is then easy to check that $\beta_{t_0}$ satisfies
\begin{equation}
	\label{eq:beta-diff}
	(  v  \partial_v^2
	+ \partial_v
	- 4 \pi^2 Q_0(\lambda)^2 v   )\beta_{t_0}(\lambda \sqrt{v})
	= 2 \pi  \lambda_{t_0}^+ \lambda_{t_0}^- e^{-2\pi v((\lambda_{t_0}^+)^2 + (\lambda_{t_0}^-)^2)}
\end{equation}
for any $\lambda \in V_0(\Rb) \backslash (c(t_0) \cup c(t_0)^\perp)$. 
This equation proves the following theorem of Zwegers.
\begin{thm}
  \label{thm:Zw-diff}[Theorem 2.7 of \cite{Zw12}]
  In the notation above, we have
\footnote{  Note that the unary theta series as defined here and the one defined by Zwegers differ by a factor of 4, i.e.\ $4\theta_{L+h}^{(1, 1)}(\tau, t)$ equals $\vartheta_{h,0}^c(\tau)$.}
  \begin{equation}
    \label{eq:diff-eq1}
    \begin{split}
      \tD  \varphi^{c_i}_{L + h}(\tau)
      &= 2\pi\theta_{L+h}^{(1, 1)}(\tau, t_i),\\
      \tD  \hv^{c_1, c_2}_{L + h}(\tau)
      &=2\pi(\theta_{L+h}^{(1, 1)}(\tau, t_1)-\theta_{L+h}^{(1, 1)}(\tau, t_2)).
    \end{split}
  \end{equation}
\end{thm}

\begin{rmk}
When $t_0^2 \in F^1$, the function $\theta_{L+h}^{(1, 1)}(\tau, t_0) $ is in $S_{3/2} \otimes v^{3/2}\overline{S_{3/2}}$ by \eqref{eq:theta-split}.
\end{rmk}

\begin{exmp}\label{ex:cohen}
In Remark 2.2 and Section 6 of \cite{Zw12} it is re-proved that Cohen's function $\varphi_0$ is a Maass form of weight 0 by relating it to $\hat\Phi_{a, b}^{c_1, c_2}$ for suitable $a, b, c_1, c_2$.

Zwegers considers the quadratic space $(\Zb^2,Q')$ with $Q'((m,n))=\frac12 (3m^2-2n^2)$. Further, the indefinite theta series depends on the characteristics $a=\lp\begin{smallmatrix}
	1/6\\0
\end{smallmatrix}\rp$ and $b=\lp\begin{smallmatrix}
1/6\\1/4
\end{smallmatrix}\rp$.

Under the identification $(m,n)\mapsto 3m+\sqrt{6}n$ we have $(\Zb^2,Q')\simeq (\af,Q)$ with $\af=3\Zb+\sqrt{6}\Zb$ an integral ideal in the ring of integers $\Oc$ of $F = \Qb(\sqrt{6})$ and $Q=\Nm/6$. The sublattice $2\Zb^2$ is then isometric to $L:= L_{\af, 2} = 2\af = 6\Zb+2\sqrt{6}\Zb$.
Its dual lattice is $\af\df^{-1} = \frac{1}{2}\Zb + \frac{\sqrt{6}}{4} \Zb$, 
where $\df = 2\sqrt3 \Oc$. 
We split
\begin{align*}
  a+\Zb^2 &= \lp a + 2\Zb^2 \rp \sqcup
            \lp a + \binom{1}{0} + 2\Zb^2 \rp \sqcup \lp a +  \binom{0}{1} + 2\Zb^2 \rp \sqcup \lp a +  \binom{1}{1} + 2\Zb^2 \rp \\
           & \simeq \lp \frac12+L \rp\sqcup \lp \frac72+L \rp \sqcup \lp \frac12+\sqrt{6}+L \rp\sqcup \lp \frac72+\sqrt{6}+L \rp.
\end{align*}
We thus consider the theta series associated to $L=L_{\af, 2}$.
Let $t_1=1$, $t_2= \varepsilon := 5+2\sqrt{6}$.
For
\begin{equation}
  \label{eq:hs}
h \in \{1/2, 7/2, 1/2 + \sqrt{6}, 7/2 + \sqrt{6}\}  
\end{equation}
 and $i = 1, 2$, the theta function $\theta^{(1, 1)}_{L+h}(\tau, t_i)$ vanishes identically since  the automorphism $\lambda \mapsto \lambda'$ on $L+h$ creates a minus sign.
That means $\hv^{c_1, c_2}_{L+h} = \hv^{c_1, c_2, +}_{L+h}$ is a Maass form by Theorem \ref{thm:Zw-diff}.
Furthermore, we have the identity
\begin{equation}
  \label{eq:Maass-id}
  \hv^{c_1, c_2}_{L+h}
  =  \hv^{c_1, c_2, +}_{L+h}
=  \frac{\vartheta_{L+h}}{4} =   \frac{\vartheta_{L+h+3+\sqrt{6}}}{4} 
\end{equation}
for all $h$ in \eqref{eq:hs}.

Notice that if $\lambda \in L + h$ for $h$ in \eqref{eq:hs}, then $\bfrak = (2\lambda)$ is an integral ideal of $\Oc$ with norm congruent to 1 modulo 24.
Conversely given such an ideal $\bfrak$, we can find a unique generator $\beta$ such that $1 \le \beta/\beta' < \varepsilon^2$ and $\beta/2 \in L + h$ for an $h$ in \eqref{eq:hs}.
In case $\beta/\beta' = 1$, i.e.\ $\beta \in \Qb$, then $\lambda = -\varep \beta/2$ also gives non-zero value for $a_\lambda(1, \varep)$. 
We list these ideals with norm less than 100 together with the generator and value $a_\lambda(1, \varep)$ below.
\begin{table}[h]
\caption{Values of $a_\lambda(1, \varep)$}
\label{table:ideals}
$\begin{array}{c|c|c|c|c|c}
&    & \frac12 &  \frac72 + \sqrt{6} &  \frac72 &  \frac12+\sqrt{6} \\ 
  \text{Norm} & \beta = 2 \lambda & & & \\\hline
   1 &   1, -\varep & \frac12 & \frac12 & & \\ \hline
   \multirow{2}{*}{23} & 1 + 2\sqrt{6} & & & & 1 \\
& -\varep(1 - 2\sqrt{6}) & & &1 & \\ \hline
   \multirow{3}{*}{25} & 7 + 2\sqrt{6} & & 1& &  \\
& -\varep(7 - 2\sqrt{6}) & 1& & & \\ 
& -5, 5\varep & & & \frac12 &\frac12\\       \hline
   \multirow{2}{*}{47} & 7 + 4\sqrt{6} & & & 1&  \\
& -\varep(7 - 4\sqrt{6}) & & & &1 \\ \hline
   {49} & 7,-7\varep & & & \frac12 &\frac12\\   \hline
   \multirow{2}{*}{71} & -(5 + 4\sqrt{6}) & & & 1&  \\
& \varep(5 - 4\sqrt{6}) & & & &1 \\ \hline
   \multirow{2}{*}{73} &   13+4\sqrt{6} & 1 & & & \\
& -\varep(13-4\sqrt{6}) & & 1 & & \\ \hline
   \multirow{4}{*}{95} & 1 + 4\sqrt{6} & 1& & &  \\
& -\varep(1 - 4\sqrt{6}) & &1 & & \\ 
& \varep(11-6\sqrt{6}) & & & 1 &\\   
& -(11-6\sqrt{6}) & & & &1 \\    \hline
   \multirow{2}{*}{97} & -(11 +2\sqrt{6}) & & & & 1 \\
& \varep(11 - 2\sqrt{6}) & & &1 & \\ \hline
 \end{array}$
\end{table}

Also, the map $\lambda \mapsto -\varepsilon \lambda'$ identifies $  \hv^{c_1, c_2,+}_{L+\frac12}(\tau)$ (resp.\  $  \hv^{c_1, c_2,+}_{L+\frac72}(\tau)$)
with $  \hv^{c_1, c_2,+}_{{L}+\frac12+\sqrt{6}}(\tau)$ (resp.\ $  \hv^{c_1, c_2,+}_{{L}+\frac72+\sqrt{6}}(\tau)$). 
Using these, we have
\begin{align*}
\sqrt{24}^{-1}  \hv^{c_1, c_2,+}_{L+\frac12}(24\tau)
  &= \dots + W_{-95}(\tau) + \frac12 W_{1}(\tau) + W_{25}(\tau) + W_{73}(\tau) + \frac12 W_{121}(\tau) + \dots,\\
\sqrt{24}^{-1}  \hv^{c_1, c_2,+}_{L+\frac72}(24\tau)
  &= \dots  +W_{-95}(\tau) + W_{-71}(\tau) + W_{-47}(\tau) + W_{-23}(\tau)\\
&\quad    + \frac12 W_{25}(\tau) + \frac12 W_{49}(\tau) +  W_{97}(\tau) + \dots,
\end{align*}
where $W_n(\tau) := \sqrt{v}\ebf(nu) K_0(2\pi |n|v)$. 
We can now write Cohen's function as
\begin{equation}
  \label{eq:cohen}
  \begin{split}
    \varphi_0(\tau)
    &=
\hv^{c_1, c_2,+}_{L+\frac12}(\tau)
- \hv^{c_1, c_2,+}_{L+\frac72}(\tau)
- \hv^{c_1, c_2,+}_{L+\frac12+\sqrt{6}}(\tau)
+  \hv^{c_1, c_2,+}_{L+\frac72+\sqrt{6}}(\tau)\\
&= \frac12 \lp \vartheta_{L+1/2}(\tau) - \vartheta_{L+7/2}(\tau)\rp.
  \end{split}
\end{equation}
\end{exmp}



\subsection{ Mock Maass Theta Functions II}
\label{subsec:mMt2}
In this section, we will further decompose the function $\hv^{c_1, c_2}_{L+h}$ into the difference of two functions that depend only on $c_1$ and $c_2$ respectively.

For $c_0 = c(t_0)$, consider
\begin{equation}
  \label{eq:tv}
  \begin{split}
\tv^{c_0}_{L+h}(\tau)
&:=  \int^{ t_0 \varepsilon_L}_{ t_0 } \theta_{L+h}(\tau, t) \log (t) \frac{dt}{t}
=    {\sqrt{v}}
\sum_{\lambda \in L + h}
\ebf\lp
  Q_0(\lambda)u\rp
  \tbe_{t_0}(\lambda \sqrt{v}),\\
    \tbe_{t_0}(w)
&:=
e^{-2\pi Q_0(w)}
  \int^{t_0 \varepsilon_L}_{t_0}
e^{-4 \pi  (w^-_t)^2 }
\log( t)  \frac{dt}{t},
  \end{split}
\end{equation}
where $L \subset V_0(\Rb)$ is an even, integral, anisotropic lattice.
Now for $\Lambda = \Gamma_L \cdot \lambda \in \Gamma_L \backslash L + h$,
we define
\begin{equation}
  \label{eq:cLam}
  \tc_{t_0}(\Lambda) :=
  \begin{cases}
\frac{\log (t_0^2 \varep_L)}{2}  + 
\log \varep_L \cdot \Bb_1
\lp \frac{\log (|\lambda/\lambda'| t_0^{-2})}{2\log \varep_L} \rp
& \text{ if } Q_0(\Lambda) \neq 0,\\
\frac12 \log (\varep_L ) \log(t^2_0 \varepsilon_L)
 & \text{ if }Q_0(\Lambda) = 0.
  \end{cases}
\end{equation}
Here $\Bb_1 : \Rb/\Zb \to \Rb$ is the first periodic Bernoulli polynomial defined by $\Bb_1(x) = x - \frac{\delta_{x \neq 0}}{2}$ for $0 \le x < 1$.
Note that $Q_0(\Lambda) = 0$ if and only if $\lambda$ is trivial since $L$ is anisotropic.
Then we have the following result.
\begin{prop}
  \label{prop:compare}
  For any non-trivial $\lambda_0 \in L + h$ and $n \in \Zb$, denote $\lambda_n := \lambda_0 \varepsilon_L^n$.
  Let $\Lambda := \Gamma_L \cdot \lambda_0$. 
Then we have
  \begin{equation}
    \label{eq:compare}
    \sum_{n \in \Zb} \tbe_{t_0} (\lambda_n \sqrt{v})
    =
  - \log \varepsilon_L  \sum_{n \in \Zb} \beta_{t_0} (\lambda_n \sqrt{v})
  + \tc_{t_0}(\Lambda) \cdot
  K_0(2\pi |Q_0(\lambda_0)| v).
  \end{equation}
\end{prop}
\begin{proof}
  Denote $m:= Q_0(\lambda)$.
  From \eqref{eq:beta-diff}, we have
  $$
\lp  v^2  \partial_v^2
  + v\partial_v
  - 4 \pi^2 m^2 v^2 \rp
\lp \sum_{n \in \Zb}  \beta_{t_0}(\lambda_n \sqrt{v}) \rp
= 2 \pi v
\sum_{n \in \Zb}
(\lambda_n)_{t_0}^+ (\lambda_n)_{t_0}^- e^{-2\pi v(((\lambda_n)_{t_0}^+)^2 + ((\lambda_n)_{t_0}^-)^2)}.
$$
On the other hand, we have
\begin{align*}
&\lp  v^2  \partial_v^2
  + v\partial_v
  - 4 \pi^2 m^2 v^2 \rp
  \tbe_{t_0}(\lambda_n \sqrt{v})
  = -2 \pi v \log \varepsilon_L
   ((\lambda_{n-1})_{t_0}^+ (\lambda_{n-1})_{t_0}^-)
        e^{-2 \pi v (((\lambda_{n-1})^+_{t_0})^2 +   ((\lambda_{n-1})^-_{t_0})^2) }\\
&-  2\pi v \log (t_0)\lp  ((\lambda_{n-1})_{t_0}^+ (\lambda_{n-1})_{t_0}^-)
        e^{-2 \pi v (((\lambda_{n-1})^+_{t_0})^2 +   ((\lambda_{n-1})^-_{t_0})^2) }
    -   ((\lambda_n)_{t_0}^+ (\lambda_n)_{t_0}^-)    
    e^{-2 \pi v (((\lambda_n)^+_{t_0})^2 +   ((\lambda_n)^-_{t_0})^2) }
    \rp\\
  & - \frac{1}{4} \lp
        e^{-2 \pi v (((\lambda_{n-1})^+_{t_0})^2 +   ((\lambda_{n-1})^-_{t_0})^2) }
    -   
    e^{-2 \pi v (((\lambda_n)^+_{t_0})^2 +   ((\lambda_n)^-_{t_0})^2) }
    \rp.
\end{align*}
Therefore applying the operator $ v^2  \partial_v^2
  + v\partial_v
  - 4 \pi^2 m^2 v^2 $ to the left hand side of \eqref{eq:compare} gives us
  \begin{align*}
\lp     v^2  \partial_v^2
  + v\partial_v
    - 4 \pi^2 m^2 v^2 \rp
    \lp     \sum_{n \in \Zb} \tbe_{t_0} (\lambda_n \sqrt{v})\rp
    &=
 -2 \pi v \log \varepsilon_L
\sum_{n \in \Zb}
(\lambda_n)_{t_0}^+ (\lambda_n)_{t_0}^- e^{-2\pi v(((\lambda_n)_{t_0}^+)^2 + ((\lambda_n)_{t_0}^-)^2)}
  \end{align*}
  and
  \begin{align*}
\lp     v^2  \partial_v^2
  + v\partial_v
    - 4 \pi^2 m^2 v^2 \rp
    \lp     \sum_{n \in \Zb} \tbe_{t_0} (\lambda_n \sqrt{v})
    + \log \varepsilon_L
\sum_{n \in \Zb} \beta_{t_0} (\lambda_n \sqrt{v})    
    \rp = 0.
  \end{align*}
  The solutions to the differential equation $(v^2 \partial_v + v \partial_v - 4 \pi m^2 v^2 )f = 0$ are of the form
  $$
f = b \cdot I_0(2\pi|m|v) + c \cdot K_0(2\pi|m|v)
$$
with $I_0$ and $K_0$ the modified Bessel functions and $b, c$ constants.
Since $\beta(\lambda_0 \sqrt{v})$ and $\tbe(\lambda_0 \sqrt{v})$ both decay as $v \to \infty$, we have
$$
     \sum_{n \in \Zb} \tbe_{t_0} (\lambda_n \sqrt{v})
    + \log \varepsilon_L
    \sum_{n \in \Zb} \beta_{t_0} (\lambda_n \sqrt{v})
    = c \cdot K_0(2\pi |m| v). 
    $$
    Using the bound \eqref{eq:K0bound} in the proof of Lemma \ref{lemma:asymp} and equation \eqref{eq:beta2}, we have
    $$
\lim_{v \to \infty} \frac{    \sum_{n \in \Zb} \beta_{t_0} (\lambda_n \sqrt{v})}{K_0(2\pi |m| v)} = 0.
$$

On the other hand, a change of variable gives us
\begin{align*}
  \sum_{n \in \Zb}& \tbe_{t_0} (\lambda_n \sqrt{v})
  =
    \sum_{n\in\Zb}\int_{t_0|\lambda'_n/\lambda_n|^{\frac12}}^{t_0|\lambda'_{n-1}/\lambda_{n-1}|^{\frac12}}e^{-\pi v |m| (r^2 + r^{-2})}\log ( r |\lambda_n/\lambda'_n|^{\frac12}) \frac{dr}{r}\\
  &=
    \frac12\log|\lambda_0/\lambda'_0| K_0(2\pi |m| v)
    +
    \log\varep_L\sum_{n\in\Zb} n\int_{t_0|\lambda'_n/\lambda_n|^{\frac12}}^{t_0|\lambda'_{n-1}/\lambda_{n-1}|^{\frac12}}
    e^{-\pi |m| v(r^2+r^{-2})} \frac{dr}{r}.
\end{align*}
The second sum can be evaluated as
\begin{align*}
  \sum_{n\in\Zb}
  &n\int_{t_0|\lambda'_n/\lambda_n|^{\frac12}}^{t_0|\lambda'_{n-1}/\lambda_{n-1}|^{\frac12}}
    e^{-\pi |m| v(r^2+r^{-2})} \frac{dr}{r}\\
  &=   \sum_{n\ge 1}
    \lp \int_{0}^{t_0|\lambda'_{n-1}/\lambda_{n-1}|^{\frac12}}
    e^{-\pi |m| v(r^2+r^{-2})} \frac{dr}{r}
-  \int^{\infty}_{t_0|\lambda'_{-n}/\lambda_{-n}|^{\frac12}}
    e^{-\pi |m| v(r^2+r^{-2})} \frac{dr}{r}
    \rp\\
  &=
\frac12
    \sum_{n\ge 1}
    \lp
   K_0(2\pi |m| v, \log (t_0^{-2}|\lambda_{n-1}/\lambda_{n-1}'|)  )
-  K_0(2\pi |m| v, \log(t_0^2 |\lambda_{-n}'/\lambda_{-n}|))
    \rp.
\end{align*}
We now choose $\lambda_0$ such that $t_0^2 \le |\lambda_0 /\lambda_0'| < t_0^2 \varepsilon_L^2$, which implies that $t_0^{-2}|\lambda_0/\lambda_0'| \ge 1$, $t_0^{-2}|\lambda_n/\lambda_n'| > 1$ and  $t_0^{2}|\lambda_{-n}'/\lambda_{-n}| > 1$
for all $n \ge 1$. 
Applying Lemma \ref{lemma:asymp} and equation \eqref{eq:K0bound} gives us
\begin{align*}
  \lim_{v \to \infty} K_0(2\pi |m| v)^{-1}
    \sum_{n\in\Zb}
  n\int_{t_0|\lambda'_n/\lambda_n|^{\frac12}}^{t_0|\lambda'_{n-1}/\lambda_{n-1}|^{\frac12}}
  e^{-\pi |m| v(r^2+r^{-2})} \frac{dr}{r}
  =
  \frac12
  \begin{cases}
    1& \text{ if } t_0^2 = |\lambda_0/\lambda_0'|,\\
    0&\text{ otherwise.}
  \end{cases}
\end{align*}
This finishes the proof.
\end{proof}
Now, we can define the harmonic generating series
\begin{equation}
  \label{eq:tvar+}
  \tv^{c_0, +}_{L+h}(\tau) := \sqrt{v}\sum_{\Lambda \in \Gamma_L\backslash L + h}
  \tc_{t_0}(\Lambda) \ebf(Q_0(\Lambda) u) K_0(2\pi |Q_0(\Lambda)| v). 
\end{equation}
As an immediate consequence of Proposition \ref{prop:compare}, we have the following result.
\begin{thm}
  \label{thm:main}
  Let $c_0 = c(t_0)$ with $t_0 > 0$. Then the real-analytic modular form $\tv^{c_0}_{L+h}$ has the decomposition
  \begin{equation}
    \label{eq:main}
    \tv^{c_0}_{L+h}(\tau) =     \tv^{c_0, +}_{L+h}(\tau) - \log(\varepsilon_L) \varphi^{c_0}_{L+h}(\tau)
  \end{equation}
  into the sum of a harmonic and a non-harmonic part.
  Furthermore for $c_i = c(t_i)$, the mock Maass theta function $\hv^{c_1, c_2, +}_{L+h}$ defined in \eqref{eq:vt} and its completion $\hv^{c_1, c_2}_{L+h}$ can be written as
  \begin{equation}
    \label{eq:vt-decomp}
    \begin{split}
      - \log(\varepsilon_L)     \hv^{c_1, c_2, +}_{L+h}(\tau)
      &=
        \tv^{c_1, +}_{L+h}(\tau)    - \tv^{c_2, +}_{L+h}(\tau),\\
      - \log(\varepsilon_L)     \hv^{c_1, c_2}_{L+h}(\tau)
      &=
        \tv^{c_1}_{L+h}(\tau)    - \tv^{c_2}_{L+h}(\tau).
    \end{split}
\end{equation}
Under the differential operator $\tD$, we have
\begin{equation}
  \label{eq:tD-tv}
  \tD \tv^{c_0}_{L+h}(\tau) = -\log (\varep_L) 2\pi \theta^{(1, 1)}_{L+h}(\tau, t_0).
\end{equation}
\end{thm}

\begin{proof}
  Equation \eqref{eq:main} follows directly from Proposition \ref{prop:compare}.
  To prove the first equation in \eqref{eq:vt-decomp}, we first check that the constant terms are both $-\log (\varep_L) \log(t_2/t_1)$.
Then  it suffices to check 
\begin{equation}
  \label{eq:compare2}
  -\log (\varep_L) \sum_{n \in \Zb} a_{\lambda_n}(t_1, t_2)
  = \tc_{t_1}(\Lambda)  -  \tc_{t_2}(\Lambda) 
\end{equation}
  for all non-trivial $\lambda_0 \in L^*$, where we have denoted $\lambda_n := \lambda_0 \varep_L^n$ and $\Lambda := \Gamma_L \cdot \lambda_0$.
  We can choose a unique $\lambda_0$ within the orbit $\Lambda$ such that
  $$
t_1 \le |\lambda_0/\lambda_0'|^{1/2} < t_1 \varep_L. 
$$
  Let $m := Q_0(\lambda_0) \in \Qb^\times$  and  $N \in \Nb_0$ such that
  $$
  t_1 \varep_L^{N} \le t_2 < t_1 \varep_L^{N+1}.
  $$
  Suppose $|\lambda_N/\lambda_N'| \le t_2$.
  Then $t_2 < t_1 \varep_L^{N+1} \le |\lambda_0/\lambda_0'|^{1/2} \varep_L^{N+1} = |\lambda_{N+1}/\lambda_{N+1}'|^{1/2}$, and
 $t_1 \le |\lambda_n/\lambda_n'|^{1/2} \le t_2$ precisely when $0 \le n \le N$. 
  The two equalities could only hold for $n = 0, N$, and only one holds when $N = 0$ since $t_2 > t_1$.
  Therefore, we have
  \begin{align*}
      \sum_{n \in \Zb} a_{\lambda_n}(t_1, t_2)          
      &=
        \#       \{n \in \Zb: t_1 < |\lambda_n/\lambda_n'|^{1/2} < t_2\}
        + \frac12 \{n \in \Zb: t_i = |\lambda_n/\lambda_n'|^{1/2} \text{ for } i = 1 \text{ or } 2\}\\
      &= N + 1- \frac12 \lp \delta_{t_1^2 = |\lambda_0/\lambda_0'|}  +
        \delta_{t_2^2 = |\lambda_N/\lambda_N'|}  \rp. 
  \end{align*}
  On the other hand, we have
  \begin{align*}
- &\frac{    \tc_{t_1}(\Lambda)  -  \tc_{t_2}(\Lambda)}{\log \varep_L}
    = \frac{\log(t_2/t_1)}{\log \varep_L} -
      \lp
      \Bb_1      \lp \frac{\log (|\lambda_0/\lambda_0'| t_1^{-2})}{2\log \varep_L} \rp
      -
      \Bb_1      \lp \frac{\log (|\lambda_N/\lambda_N'| t_2^{-2})}{2\log \varep_L} \rp            \rp\\
    &=  \frac{\log(t_2/t_1)}{\log \varep_L}
      - \lp
\frac{\log (|\lambda_0/\lambda_0'| t_1^{-2})}{2\log \varep_L}  
      - \frac{\delta_{t_1^2 \neq |\lambda_0/\lambda_0'|}}{2} 
      \rp
     +
      \lp
 \frac{\log (|\lambda_N/\lambda_N'| t_2^{-2})}{2\log \varep_L}  
      + \frac{\delta_{t_2^2 \neq |\lambda_N/\lambda_N'|}}{2} 
      \rp\\
  &= N + \frac12 \lp  \delta_{t_1^2 \neq |\lambda_0/\lambda_0'|} + \delta_{t_2^2 \neq |\lambda_N/\lambda_N'|}\rp
    = N + 1 - \frac12 \lp  \delta_{t_1^2 = |\lambda_0/\lambda_0'|} + \delta_{t_2^2 = |\lambda_N/\lambda_N'|}\rp
  \end{align*}
since $1 \le t_1^{-1}|\lambda_0/\lambda_0'|^{1/2} < \varepsilon_L$ and 
$\varep_L^{-1} < t_2^{-1}|\lambda_N/\lambda_N'|^{1/2} \le 1$.
The case $t_2 < |\lambda_N/\lambda_N'|$ can be handled similarly.
The second equation in \eqref{eq:vt-decomp} follows from \eqref{eq:har-decomp}, \eqref{eq:main} and the first equation in \eqref{eq:vt-decomp}.
Finally, equation \eqref{eq:tD-tv} follows from \eqref{eq:main} and \eqref{eq:diff-eq1}.
\end{proof}

\begin{exmp}
\label{ex:cohen2}
  Let $L = L_{\af, M}$ be the lattice in Example \ref{ex:cohen} and recall that we chose $t_1=1$ and $t_2=\varep=5+2\sqrt{6}$.
  For $i = 1, 2$ and $h$ in \eqref{eq:hs}, the function $\theta^{(1, 1)}_{L+h}(\tau, t_i)$ vanishes identically, and $\tv^{c_i}_{{L}+h} = \tv^{c_i, +}_{{L}+h}$ is a Maass form.
  We compute the Fourier expansions of the theta series $ \tv^{c_i, +}_{ {L}+h}$ given in  \eqref{eq:tvar+} by calculating $\tc_i(\Lambda)$ as defined in \eqref{eq:cLam}.
  Note that $\Gamma_L=\langle\varep_L\rangle$, where $\varep_L=\varep^2$. We thus have
\begin{align*}
  	(\log \varep_L)^{-1}\tc_1(\Lambda)&=\frac12+\Bb_1\lp\frac{\log(|\lambda/\lambda'|)}{4\log(\varep)}\rp,\\
  	(\log \varep_L)^{-1}\tc_2(\Lambda)&=1+\Bb_1\lp\frac{\log(|\lambda/\lambda'|)}{4\log(\varep)}-\frac12\rp.
\end{align*}
Since $\varep_L$ preserves the set $L+h$ for all $h$ given above, it is sufficient to compute $c_i(\Lambda)$ for one representative $\Lambda \in \Gamma_L\backslash L + h$.
To evaluate the periodic Bernoulli polynomials, note that
\begin{equation}
  \label{eq:Bernoulli}
  \Bb_1(x + n) + \Bb_1(-x) = 0
\end{equation}
for all $x \in \Rb, n \in \Zb$.
Using table \ref{table:ideals}, we obtain for $i = 1, 2$ and $h$ as in \eqref{eq:hs}
\begin{equation}
  \label{eq:tvMaass}
  \tv_{L+h}^{c_i, +}(\tau)
  =
i \frac{\vartheta_{L+h}(\tau)}{8} .
\end{equation}
Now it is easy to see that
  \begin{align*}
    \varphi_0 &=
  - (i \log \varep_L)^{-1}\lp\tv^{c_i}_{L+\frac12 {}}
- \tv^{c_i}_{L+\frac72 {}}
- \tv^{c_i}_{L+\frac12+\sqrt{6}}
                +  \tv^{c_i}_{L+\frac72+\sqrt{6}}\rp\\
              &=
  - ( \log \varep_L)^{-1}\lp\tv^{c_2}_{L+\frac12 {}}
- \tv^{c_2}_{L+\frac72 {}}
- \tv^{c_2}_{L+\frac12+\sqrt{6}}
                +  \tv^{c_2}_{L+\frac72+\sqrt{6}}\rp\\
&\quad  + ( \log \varep_L)^{-1}\lp\tv^{c_1}_{L+\frac12 {}}
- \tv^{c_1}_{L+\frac72 {}}
- \tv^{c_1}_{L+\frac12+\sqrt{6}}
                +  \tv^{c_1}_{L+\frac72+\sqrt{6}}\rp.
  \end{align*}
  \end{exmp}
  \begin{exmp}
\label{ex:non-trivial}
    Let $F$ be the same as in Example \ref{ex:cohen}. Take $L = L_{\af, M}$ with $\af = \Oc = \Zb + \sqrt{6}\Zb$ and $M = 2$.
    Then $L^* = \frac12 \Zb + \frac{\sqrt6}{12}\Zb, \varep_L = \varep^2 = (5 + 2\sqrt6)^2 = 49 + 20\sqrt6$,
    and it is easy to check that
    $$
    \theta^{(1, 1)}_{L + h}(\tau, t_0)
    = - \frac{\sqrt6}{24} \eta(\tau)^3 \overline{g(\tau)}v^{3/2}
    $$
    for $h =  \frac12 + \frac{\sqrt6}{12}$ and $t_0 = 1$, where $$g(\tau) = q^{1/48}(1 - 23q^{11} + 25q^{13} - 47q^{46} + 49q^{50} - 71q^{105} + 73q^{111} - 95q^{188} + 97q^{196} + O(q^{200}))$$ is a modular form of weight $3/2$. 
    Using the isomorphism
    $$
\Gamma_L \backslash (L + h \sqcup L + h') \cong \{\bfrak = (\beta) \subset \Oc: \Nm(\beta) \equiv -5 \bmod{48}\},~ [\lambda] \mapsto (2\sqrt{6}\lambda),
$$
with $h' =   \frac12 - \frac{\sqrt6}{12}$, 
we can write
\begin{align*}
  \sqrt{12}^{-1}
  &\tv_{L+h}^{c_0, +}(48\tau)
  =
    \dots + \lp 3 \log \varep + \log \frac{29+6\sqrt6}{25} \rp W_{-235}(\tau)\\
    &+ \lp \log \varep + \log \frac{155+28\sqrt6}{139}\rp W_{-139}(\tau)
    + \lp 2\log \varep + \log \frac{55 + 14\sqrt6}{43} \rp W_{-43}(\tau)\\
      &+ \log \frac{7+2\sqrt6}{5}W_{5}(\tau) + \lp 4\log \varep+ \frac{55-6\sqrt6}{53}\rp W_{53}(\tau) \\
  &+ \lp 3 \log \varep + \log \frac{155 - 48\sqrt6}{101}\rp W_{101}(\tau)
+ \log \frac{151+10\sqrt6}{149}W_{149}(\tau) +    \dots.
\end{align*}
\end{exmp}
\section{Action under $\tD$.}
In this section, we will generalize the mock Maass theta functions studied in \ref{subsec:mMt2}, and study their images under the differential operator $\tD$.
In the notation of Sections \ref{subsec:mock1} and \ref{subsec:mMt2}, consider
\begin{equation}
  \label{eq:higher1}
  \hv^{c_1, c_2}_{L+h}(\tau; P)
:=  \int^{ t_2}_{ t_1} \theta_{L+h}(\tau, t)P(\log t) \frac{dt}{t}
= \int^{ \nu_2}_{ \nu_1} \theta_{L+h}(\tau,  e^{\nu}) P(\nu) d\nu
\end{equation}
for any polynomial $P(X) \in \Cb[X]$.
This defines a real-analytic modular function in $\tau$.
When $P$ is the constant function, this is just the function $\hv^{c_1, c_2}_{L+h}(\tau)$ defined in \eqref{eq:int-c1-c2}.
When $P(X) = X$ and $t_i = \varepsilon_L^{i-1} t_0$, this is the function $\tv^{c_0}_{L+h}$ from \eqref{eq:tv}. 
Just as the function $\hv^{c_1, c_2}_{L+h}$ in Section \ref{subsec:mock1}, the generalization $\hv^{c_1, c_2}_{L+h}(\tau; P)$ has the following Fourier expansion
\begin{equation}
  \label{eq:FEhvP}
  \hv^{c_1, c_2}_{L + h} (\tau;P)
  ={\sqrt{v}}
\sum_{\lambda \in L + h}
\ebf\lp
Q_0(\lambda)\tau\rp
\int^{t_2}_{t_1}
e^{-\pi v B_0(\lambda, Z^-_t)^2 }
P(\log t)
\frac{dt}{t}.
\end{equation}
When $t_2 = \varepsilon_L t_1$, we simply write $\hv^{c_1}_{L+h}$ for  $\hv^{c_1, c_2}_{L+h}$.

To understand the action of $\tD$ on this function, we first record a standard differential operator relation satisfied by the Siegel theta kernel.
\begin{lemma}
  \label{lemma:diffop}
In the notation above, we have
\begin{equation}
  \label{eq:diffeq}
  4\tD \theta_{L + h}(\tau, t) = (t \partial_t)^2  \theta_{L + h}(\tau, t).
\end{equation}
\end{lemma}
\begin{proof}
	One can show that the function $f(\tau, t)$ defined by
	\begin{equation*}
		\begin{split}
			f(\tau,t)&:= \phi_\tau(\lambda_t^+, \lambda_t^-),~
			\phi_\tau(x, y) :=
			\sqrt{v}  \ebf \lp {x^2}{} \tau - {y^2}{} \overline{\tau} \rp,     
		\end{split}
	\end{equation*}
	satisfies the identity
	\begin{align*}
		(t \partial_t)^2  f(\tau, t)&=\lp\lp2\pi v\rp^2\lp(\lambda_1 t^{-1})^2-(\lambda_2t)^2\rp^2-4\pi v\lp(\lambda_1 t^{-1})^2+(\lambda_2t)^2\rp\rp f(\tau, t)\\
		&=4\lp\lp 4\pi v \lambda_t^+ \lambda_t^-\rp^2-2\pi v \lp (\lambda_t^+)^2+(\lambda_t^-)^2\rp\rp f(\tau, t).
	\end{align*}
	On the other hand, a straightforward calculation yields
	\begin{equation*}
		\tD \phi_{\tau}(x,y)=\lp (4\pi vxy)^2-2\pi v(x^2+y^2) \rp\phi_{\tau}(x,y).
	\end{equation*}
	By the definition of the theta series in \eqref{eq:theta_explicit}, we deduce that \eqref{eq:diffeq} holds.
\end{proof}

Now the image of $\hv^{c_1, c_2}_{L+h}$ under the operator $\tD$ is given by the following lemma.
\begin{lemma}
  \label{lemma:tDP}
  We have
  \begin{multline*}
  	4\tD \hv^{c_1, c_2}_{L+h}(\tau; P)= 
  	(t \partial_t \theta_{L+h})(\tau,  t_2) P(\nu_2) -  (t \partial_t \theta_{L+h})(\tau,  t_1) P(\nu_1)\\
  	\quad -
  	\lp 
  	\theta_{L+h}(\tau,  t_2) P'(\nu_2) -   \theta_{L+h}(\tau,  t_1) P'(\nu_1)
  	\rp + \hv^{c_1, c_2}_{L+h}(\tau; P'').
  \end{multline*}
\end{lemma}
\begin{rmk}\label{rem:theta_id}
  Note that $  	t \partial_t \theta_{L+h}(\tau,  t) = - 8\pi \theta^{(1, 1)}_{L+h}(\tau,  t)$. 
\end{rmk}
\begin{proof}
We apply Lemma \ref{lemma:diffop} to obtain
\begin{align*}
	4\tD &\hv^{c_1, c_2}_{L+h}(\tau; P)
	= 4\int^{ \nu_2}_{ \nu_1} \tD \theta_{L+h}(\tau,  e^{\nu}) P(\nu) d\nu 
	= \int^{ \nu_2}_{ \nu_1} (\partial_\nu^2 \theta_{L+h}(\tau,  e^{\nu})) P(\nu) d\nu \\
	&= 
	(t \partial_t \theta_{L+h})(\tau,  t_2) P(\nu_2) -  (t \partial_t \theta_{L+h})(\tau,  t_1) P(\nu_1)
	-  \int^{ \nu_2}_{ \nu_1} (\partial_\nu \theta_{L+h})(\tau,  e^{\nu}) P'(\nu) d\nu \\
	&= 
	(t \partial_t \theta_{L+h})(\tau,  t_2) P(\nu_2) -  (t \partial_t \theta_{L+h})(\tau,  t_1) P(\nu_1)
	-
	\lp 
	\theta_{L+h}(\tau,  t_2) P'(\nu_2) -   \theta_{L+h}(\tau,  t_1) P'(\nu_1)
	\rp\\
	&\quad + \hv^{c_1, c_2}_{L+h}(\tau; P'').\qedhere
\end{align*}
\end{proof}
We can now specialize $P$ and $c_i$ to obtain the following corollary, which immediately follows from equation \eqref{eq:tD-tv}, where the image of $\tv^{c_i}_{L+h}$ is given, and Remark \ref{rem:theta_id}.
\begin{cor}
  \label{cor:higher}
  For   $c_i = c(t_i)$ with $t_i > 0$, we have
  \begin{equation}
    \label{eq:higher}
    \begin{split}
    	 4\tD
    	 &\lp  \hv^{c_1, c_2}_{L+h}(\tau; X)
    	+ \frac{\log t_1  }{\log \varep_L} \tv^{c_1}_{L+h}(\tau)    
    	- \frac{\log t_2  }{\log \varep_L} \tv^{c_2}_{L+h}(\tau)
    	\rp
    	= \theta^{}_{L+h}(\tau, t_1) - \theta^{}_{L+h}(\tau, t_2),\\
         \tD
     &\lp   \hv^{c_0}_{L+h}(\tau; X^2)
    - 4 \log (t_0^2 \varep_L) \tv^{c_0}_{L+h}(\tau)    
    \rp
    = - 2\log \varep_L \theta^{}_{L+h}(\tau, t_0) + 2\vartheta^{}_{L+h}(\tau).
    \end{split}
  \end{equation}
\end{cor}
Using the same procedure as in Section \ref{subsec:mMt2}, 
one can calculate the Fourier coefficients of the harmonic part of $\hv^{c_0}_{L+h}(\tau; P)$ for any polynomial $P$.
From \eqref{eq:higher}, we see that in order to construct $\tD$-preimage of $\theta_{L+h}$, one would need to construct  $\tD$-preimage of $\vartheta_{L+h}$, which seems to require new ideas.
It would be very interesting to study the Fourier coefficients of such preimage. 

\bibliography{ref-mock-Maass}{}

\begin{thebibliography}{Maa49}

\bibitem[Bor98]{Bo98}
Richard~E. Borcherds.
\newblock Automorphic forms with singularities on {G}rassmannians.
\newblock {\em Invent. Math.}, 132:491--562, 1998.

\bibitem[CL20]{CL20}
Pierre Charollois and Yingkun Li.
\newblock Harmonic {M}aass forms associated to real quadratic fields.
\newblock {\em J. Eur. Math. Soc.}, 22(4):1115--1148, 2020.

\bibitem[Coh88]{Co88}
Henri Cohen.
\newblock {q}-identities for {M}aass waveforms.
\newblock {\em Invent. Math.}, 91(3):409--422, 1988.

\bibitem[FK17]{FK17}
Jens Funke and Stephen~S. Kudla.
\newblock Mock modular forms and geometric theta functions for indefinite
  quadratic forms.
\newblock {\em J. Phys. A}, 50(40):404001, 19, 2017.

\bibitem[Li17]{Li17}
Yingkun Li.
\newblock Harmonic {E}isenstein series of weight one.
\newblock In {\em L-functions and automorphic forms}, volume~10 of {\em
  Contrib. Math. Comput. Sci.}, pages 171--184. Springer, Cham, 2017.

\bibitem[Maa49]{Ma49}
Hans Maass.
\newblock {Ü}ber eine neue {A}rt von nichtanalytischen automorphen
  {F}unktionen und die {B}estimmung {D}irichletscher {R}eihen durch
  {F}unktionalgleichungen.
\newblock {\em Math. Ann.}, 121:141--183, 1949.

\bibitem[Sch09]{Sc09}
Nils~R. Scheithauer.
\newblock The {W}eil representation of {SL2(Z)} and some applications.
\newblock {\em Int. Math. Res. Not.}, 22(2009):1488--1545, 2009.

\bibitem[Zwe12]{Zw12}
Sander~P. Zwegers.
\newblock Mock {M}aass theta functions.
\newblock {\em Q. J. Math.}, 63(3):753--770, 2012.

\end{thebibliography}
\bibliographystyle{alpha}

\end{document}